\documentclass[10pt,a4paper]{article}
\usepackage[utf8]{inputenc}
\usepackage{amsmath,amsfonts,amssymb,amsthm}
\usepackage{xcolor}

\newtheorem{theorem}{\protect\theoremname}[section]
\newtheorem{definition}[theorem]{\protect\definitionname}
\newtheorem{lemma}[theorem]{\protect\lemmaname}
\newtheorem{proposition}[theorem]{\protect\propositionname}

\newtheorem{corollary}[theorem]{\protect\corollaryname}
\newtheorem{remark}[theorem]{\protect\remarkname}
\newtheorem{problem}[theorem]{\protect\problemname}
\newtheorem{notation}[theorem]{\protect\notationname}

\providecommand{\corollaryname}{Corollary}
\providecommand{\claimname}{Claim}
\providecommand{\definitionname}{Definition}
\providecommand{\lemmaname}{Lemma}
\providecommand{\notationname}{Notation}
\providecommand{\remarkname}{Remark}
\providecommand{\problemname}{Problem}
\providecommand{\propositionname}{Proposition}
\providecommand{\examplename}{Example}
\providecommand{\theoremname}{Theorem}

\newcommand{\N}{\mathbb{N}}

\newcommand{\slantfrac}[2]{\,^#1\!/_#2}
\newcommand{\mc}{\mathcal}

\newcommand{\fsd}{F_{\sigma\delta}}

\newcounter{vkNoteCounter}

\begin{document}
\begin{titlepage}
\thispagestyle{empty}
\title{Absolute $F_{\sigma\delta}$ spaces}
\date{}
\author{Ondřej Kalenda\thanks{This work was supported by the research grant GAČR 17-00941S. The second author was further supported by the research grant GA UK No. 915.} \qquad \qquad Vojtěch Kovařík\footnotemark[1] \\
	\emph{kalenda@karlin.mff.cuni.cz \quad kovarikv@karlin.mff.cuni.cz} 	\\
	\\
	Charles University \\
	Faculty of Mathematics and Physics \\
	Department of Mathematical Analysis \\
	\\
	Sokolovská 49, Karlín \\
	Praha 8, 186 00 \\
	Czech Republic \\
}
\maketitle
\begin{abstract}
We prove that hereditarily Lindelöf space which is $\fsd$ in some compactification is absolutely $\fsd$. In particular, this implies that any separable Banach space is absolutely $\fsd$ when equipped with the weak topology.
\end{abstract}
\end{titlepage}

\section{Introduction}
Throughout the paper, all spaces will be Tychonoff. Central to the topic of our paper is the following definition:
\begin{definition}\label{definition: F sigma delta spaces}
Let $X$ be a Tychonoff topological space. We say that $X$ is an \emph{$\fsd$ space} if there exists a compactification $cX$ of $X$, such that $X\in\fsd(cX)$.

We say that $X$ is an \emph{absolute $\fsd$ space} (or that $X$ is \emph{absolutely $\fsd$}) if $X\in\fsd(cX)$ holds for every compactification $cX$ of $X$.
\end{definition}
Note that $X$ is absolutely $\fsd$ if and only if $X\in\fsd(Y)$ holds for every Tychonoff topological space $Y$ in which $X$ is embedded.

If, in the above definition, we replace the class $\fsd$ by $G_\delta$, we get the definition of the well known concept of Čech-completeness -- however, in such a case the situation is less complicated, because every Čech-complete space is automatically absolutely $G_\delta$. Internal characterization Čech-complete spaces was given by Zdeněk Frolík, who also gave a characterization of $\fsd$ spaces in terms of complete sequences of covers (see Definition \ref{def: complete seq. of covers} below). He then asked for a description of those spaces which are \emph{absolutely} $\fsd$ (Problem 1 in \cite{frolik1963descriptive}), and this problem is still open.
 
However, Frolík did not know whether there actually exist non-absolute $\fsd$ spaces. This part of the problem was solved later by Talagrand, who found an example of such a space (\cite{talagrand1985choquet}). Thus, we formulate Frolík's problem as follows:
\begin{problem} \label{problem: 1}
Among all $\fsd$ spaces, describe those which are absolutely $\fsd$.
\end{problem}
If we are unable to completely determine the answer to Problem \ref{problem: 1}, the next best thing to do is to find a partial answer to Problem \ref{problem: 2} for as many spaces as possible.
\begin{problem} \label{problem: 2}
Let $X$ be a (possibly non-absolute) $\fsd$ space. Describe those compactifications of $X$ in which it is $\fsd$.
\end{problem}

In Section \ref{section: fsd topological spaces}, we give a partial answer to Problem \ref{problem: 2} by showing that if a $X$ is $\fsd$ in some compactification $cX$, it is automatically $\fsd$ in all larger compactifications (which is easy) and also in all compactification which are not much smaller than $cX$ (see Corollary \ref{corollary: countable quotients} for the details).

In Proposition \ref{proposition: fs disjoint covers and abs fsd} we give a partial answer to Problem \ref{problem: 1} by finding a sufficient condition for a space to be absolutely $\fsd$. This condition is similar in flavor to Frolík's characterization of $\fsd$ spaces. Applying this result, we get that hereditarily Lindelöf $\fsd$ spaces are absolutely $\fsd$ (Theorem \ref{theorem: countable network}) and that separable Banach spaces are absolutely $\fsd$ in the weak topology (Corollary \ref{corollary: banach spaces}).

In the rest of the introductory section we collect some known results and background information.

We could adapt Definition \ref{definition: F sigma delta spaces} for the lower classes of Borel hierarchy, where we have the following standard results. Their proof consists mostly of using the fact that continuous image of a compact space is compact.
\begin{remark}
Let $X$ be a topological space.
\begin{enumerate}
	\item $X$ is absolutely closed $\iff$ $X$ is compact.
	\item $X$ is absolutely $F_\sigma$ $\iff$ $X$ is $\sigma$-compact.
	\item $X$ is absolutely open $\iff$ $X$ is locally compact.
	\item $X$ is absolutely $G_\delta$ $\iff$ $X$ is Čech-complete.
\end{enumerate}
In the first two cases, $X$ being closed ($F_\sigma$) in some compactification automatically implies that $X$ is closed ($F_\sigma$) in every Tychonoff space where it is embedded. For open and $G_\delta$ spaces, we only get that $X$ is open ($G_\delta$) in those Tychonoff spaces where it is densely embedded.
\end{remark}

As shown in \cite{talagrand1985choquet}, not every $\fsd$ space is absolutely $\fsd$. This means that the class of $\fsd$ sets is the first one for which it makes sense to study Problem \ref{problem: 2}, which is one of the reasons for our interest in this particular class. However, Talagrand's is the only result of this kind (as far as the author knows), and not much else is known about `topologically' absolute $\fsd$ spaces. In \cite{kovarik@brooms}, topological absoluteness is studied for general $\mc F$-Borel classes, providing more examples based on Talagrand's construction and also some theoretical results.

Several authors have investigated slightly different notions of absoluteness for $\fsd$ spaces. Recall that in separable metrizable setting, $\fsd$ sets coincide with $\textrm{alg}\left(F\right)_{\sigma\delta}$ sets (where $\textrm{alg}\left(F\right)$ is the algebra generated by closed sets). As shown in \cite{holicky2003perfect}, the class of $\textrm{alg}\left(F\right)_{\sigma\delta}$ sets is absolute (in the sense that if a set is in $\textrm{alg}\left(F\right)_{\sigma\delta}(cX)$ for some compactification $cX$, it is automatically of this same class in every Tychonoff space where it is embedded).

In \cite{marciszewski1997absolute} and \cite{junnila1998characterizations}, the authors study metric spaces which are absolutely $\fsd$ `in a metric sense' - that is, $X\in\fsd (Y)$ for every \emph{metrizable} space $Y$ in which $X$ is embedded. In \cite{junnila1998characterizations}, the authors give a characterization of `metric absoluteness' for $\fsd$ spaces in terms of complete sequences of covers - namely that $X$ is absolutely $\fsd$ in the metric sense if and only if it has a complete sequence of $\sigma$-discrete covers.

Unfortunately, this result is not useful in our setting, because Talagrand's space is an example of non-metrizable space, which does have such a complete sequence, but it is not absolutely $\fsd$ (in our - topological - sense).

In \cite{kovarik@brooms}, it is shown that if a metric space is separable, its complexity is automatically absolute even in the topological sense. For $\fsd$ spaces, this is a special case of Theorem \ref{theorem: countable network}.

\section{Compactifications}\label{section: compactifications}
We recall the standard definitions of compactifications and their partial ordering. By \emph{compactification} of a topological space $X$ we understand a pair $(cX,\varphi)$, where $cX$ is a compact space and $\varphi$ is a homeomorphic embedding of $X$ onto a dense subspace  of $cX$. Symbols $cX$, $dX$ and so on will always denote compactifications of $X$. 

Compactification $(cX,\varphi)$ is said to be \emph{larger} than $(dX,\psi)$, if there exists a continuous mapping $f : cX\rightarrow dX$, such that $\psi = f \circ \varphi$. We denote this as $ cX \succeq dX $. Recall that for a given $T_{3\slantfrac{1}{2}}$ topological space $X$, its compactifications are partially ordered by $\succeq$ and Stone-Čech compactification $\beta X$ is the largest one.

Often, we encounter a situation where $X\subset cX$ and the corresponding embedding is identity. In this case, we will simply write $cX$ instead of $(cX,\textrm{id}|_X)$. Lastly, whenever we write symbols $cX$ or $dX$, we will automatically assume that they denote some compactifications of $X$. Much more about this topic can be found in many books - for a more recent one, see for example \cite{freiwald2014introduction}.

In the introduction, we defined the notion of being an $\fsd$ space and an absolute $\fsd$ space. Having defined the partial order $\succeq$ on the class of compactifications of $X$, we note the basic facts related to Problem \ref{problem: 2}. The proof of this remark consists of using the fact that continuous preimage of an $\fsd$ set is an $\fsd$ set.
\begin{remark}\label{remark: absolute complexity}
For a topological space $X$, we have the following:
\begin{itemize}
	\item $X$ is an $\fsd$ space $\iff$ $X\in\fsd(\beta X)$;
	\item $X$ is an absolute $\fsd$ space $\iff$ $X\in\fsd(cX)$ for every $cX$;
	\item $X\in\fsd (dX)$, $cX\succeq dX$ $\implies$ $X\in\fsd(cX)$.
\end{itemize}
\end{remark}

\begin{notation}
Let $X$ be a topological space, suppose that two of its compactifications satisfy $dX\preceq cX$ and that $\varphi : cX\rightarrow dX$ is the mapping which witnesses this fact. We denote 
\[ \mathcal{F}\left( cX, dX\right):=\left\lbrace \varphi^{-1} \left(x\right)\big | \ x\in dX, \ \varphi^{-1} \left(x\right) \textrm{ is not a singleton} \right\rbrace . \]
\end{notation}
In this sense, every compactification $dX$ smaller than $cX$ corresponds to some disjoint system $\mathcal{F}$ of compact subsets of $cX\setminus X$. Conversely, some disjoint systems of compact subsets of $cX\setminus X$ correspond to quotient mappings, which correspond to compactifications smaller than $cX$. Not every such system $\mc F$ corresponds to a compactification, but surely every finite (disjoint, consisting of compact subsets of $cX\setminus X$) $\mc F$ does.

\section{$\fsd$ spaces}
In this section, we will list some results related to $\fsd$ spaces.
\subsection{Banach spaces}
Unless otherwise specified, a Banach space $X$ (resp. its second dual), will always be equipped with weak (resp. $w^\star$) topology. In \cite{argyros2008talagrand}, a Banach space $X$ is said to be $F_{\sigma\delta}$ if it is is an $F_{\sigma\delta}$ subset of $X^{\star\star}$. Note that the space $X^{\star\star}$ is always $\sigma$-compact, so it is $F_\sigma$ in $\beta X^{\star\star}$. Consequently, any $\fsd$ Banach space is automatically an $\fsd$ space (in the sense of Definition \ref{definition: F sigma delta spaces}).

An important class of Banach spaces are the spaces which are weakly compactly generated (WCG). Recall that a Banach space $X$ is said to be WCG, if there exists a set $K\subset X$ which is weakly compact, such that $\textrm{span}(K)$ is dense in $(X,||\cdot||)$. Clearly all separable spaces and all reflexive spaces are WCG. The canonical example of non-separable non-reflexive WCG space is the space $c_0(\Gamma)$ for uncountable index set $\Gamma$. For more information about WCG spaces, see for example \cite{fabian2013functional}. The reason for our interest in WCG spaces is the following result (\cite[Theorem 3.2]{talagrand1979espaces}):
\begin{proposition}\label{proposition: WCG spaces are fsd}
Any WCG space is an $\fsd$ Banach space.
\end{proposition}
In fact, even every subspace of a WCG space is $\fsd$ in its second dual. Talagrand has found an example of an $\fsd$ Banach space which is not a subspace of a WCG space \cite{talagrand1979espaces}. This space belongs to a more general class of weakly $\mc K$-analytic spaces. A problem which had been open for a long time is whether every weakly $\mc K$-analytic space is an $\fsd$ Banach space. A counterexample has been found in \cite{argyros2008talagrand} (as well as some sufficient conditions for a weakly $\mc K$-analytic space to be an $\fsd$ Banach space). The problem which still remains unsolved is whether weakly $\mc K$-analytic spaces are topologically $\fsd$.

\subsection{Topological spaces}\label{section: fsd topological spaces}

\begin{proposition}\label{proposition: G delta characterization}
Suppose that $X\in F_{\sigma\delta}\left(cX\right)$ and $ dX \preceq cX $. Then $X\in F_{\sigma\delta}\left(dX\right)$ holds if and only if there exists a sequence $(H_n)_n$ of $F_\sigma$ subsets of $cX$, such that
\[ \left(\forall F\in\mathcal{F}\left(cX,dX\right)\right)\left(\exists n\in\N\right): X\subset H_n \subset cX \setminus F. \]
\end{proposition}
\begin{proof}
Denote by $\varphi$ the map witnessing that $dX\preceq cX$. \\*
$\implies $: Assume that $X=\bigcap F_n$, where the sets $F_n$ are $F_\sigma$ in $dX $. Denote $H_n:=\varphi^{-1}(F_n)$. Clearly, $H_n\subset cX$ are $F_\sigma$ sets containing $X$. Let $F$ be $\mathcal{F}\left(cX,dX\right)$, that is, $F=\varphi^{-1}(y)$ for some $y\in dX\setminus X$. By the assumption, we have $X\subset F_n \subset dX\setminus \left\{y\right\}$ for some $n\in \mathbb N$. By definition of $\varphi$, we get the desired result:
\[ X=\varphi^{-1}\left(X\right)\subset \varphi^{-1}\left(F_n\right)=H_n\subset \varphi^{-1}\left(dX\setminus \left\{y\right\}\right)=\varphi^{-1}\left(X\right)\setminus \varphi^{-1}\left(y\right)=X\setminus F. \]
$\Longleftarrow$: Let the sequence of sets $H_n\subset cX$ be as in the proposition. We know that $X=\bigcap F_n$ for some $F_\sigma$ sets $F_n$. We now receive sets $F'_n:=\varphi(F_n)$ and $H'_n:=\varphi(H_n)$, $n\in\N$, all of which are $F_\sigma$ in $dX$. Clearly, we have
\[ X\subset\bigcap F'_n \cap \bigcap {H'_n}. \]
For the converse inclusion, let $y\in dX \setminus X$. If $\varphi^{-1}(y)$ is a singleton, we have $\varphi^{-1}(y) \subset cX\setminus F_n$ for some $n\in\N$, and therefore $y\notin \varphi(F_n)=F'_n$. If $\varphi^{-1}(y)$ is not a singleton, then $\varphi^{-1}(y)\in \mathcal{F}\left(cX,dX\right)$, so there exists some $n\in\N$, such that $H_n\subset X\setminus \varphi^{-1}(y)$. In this case, we have $y\notin \varphi(H_n)=H'_n$.
\end{proof}

Since any $\fsd$ space is Lindelöf, we can make use of the following lemma, which follows immediately from \cite[Lemma 14]{spurny2006solution}.
\begin{lemma}\label{lemma: fsd spaces are Fs separated}
Let $X$ be a Lindelöf subspace of a compact space $L$. Then for every compact set $K\subset L \setminus X$, there exists $H\in F_\sigma\left(L\right)$, such that $X\subset H\subset L\setminus K$.
\end{lemma}

Once we have Lemma \ref{lemma: fsd spaces are Fs separated}, Proposition \ref{proposition: G delta characterization} yields the following corollary, which gives a partial answer to Problem \ref{problem: 2}:
\begin{corollary}\label{corollary: countable quotients}
Suppose that $X$ is an $\fsd$ subspace of $cX$ and $dX\preceq cX$. Then $X$ is $\fsd$ in $dX$ as well, provided that the family $\mathcal{F}\left(cX,dX\right)$ is at most countable.
\end{corollary}
In particular, this implies that there exists no such thing as a "minimal compactifications in which $X$ is $\fsd$" (unless, of course, $X$ is locally compact).

\section{Hereditarily Lindelöf spaces}\label{section: countable network}
In this section, we present the following main result:
\begin{theorem}\label{theorem: countable network}
Every hereditarily Lindelöf $\fsd$ space is absolutely $\fsd$.
\end{theorem}
Note that every $\fsd$ space is Lindelöf (because it is $\mc K$-analytic), but the converse implication to Theorem \ref{theorem: countable network} does not hold - that is, not every absolutely $\fsd$ space is hereditarily Lindelöf. Indeed, any compact space which is not hereditarily normal is a counterexample. The proof of Theorem \ref{theorem: countable network} will require some technical preparation, but we can state an immediate corollary for Banach spaces:
\begin{corollary}\label{corollary: banach spaces}
Every separable Banach space is absolutely $\fsd$ (when equipped with the weak topology).
\end{corollary}
\begin{proof} 
By Proposition \ref{proposition: WCG spaces are fsd}, every separable Banach space $X$ is $\fsd$. The countable basis of the norm topology of $X$ is a network for the weak topology. The proposition follows from the fact that spaces with countable network are hereditarily Lindelöf.
\end{proof}

We will need the notion of complete sequence of covers:
\begin{definition}[Complete sequence of covers]\label{def: complete seq. of covers}
Let $X$ be a topological space. \emph{Filter} on $X$ is a family of subsets of $X$, which is closed with respect to supersets and finite intersections and does not contain the empty set. A point $x\in X$ is said to be an \emph{accumulation point} of a filter $\mc F$ on $X$, if each neighborhood of $x$ intersects each element of $\mc F$.

A sequence $\left( \mc F_n \right)_{n\in\mathbb N}$ of covers of $X$ is said to be \emph{complete}, if every filter which intersects all $\mc F_n$-s has an accumulation point in $X$.
\end{definition}

The connection between this notion and our topic is explained by Proposition \ref{proposition: fsd iff complete sequence of covers}. Note that a cover of $X$ is said to be \emph{closed} (\emph{open}, $F_\sigma$, \emph{disjoint}) if it consists of sets which are closed  in $X$ (open, $F_\sigma$, disjoint). As a slight deviation from this terminology, a cover of $X$ is said to be \emph{countable} if it contains countably many elements.
\begin{proposition}\label{proposition: fsd iff complete sequence of covers}Any topological space $X$ satisfies
\begin{enumerate}
\item $X$ is Čech-complete $\iff$ $X$ has a complete sequence of open covers,
\item $\begin{aligned}[t]
X \textrm{ is } \fsd & \iff X \textrm{ has a complete sequence of countable closed covers} \\
& \iff 	X \textrm{ has a complete sequence of countable } F_\sigma \textrm{ covers},
\end{aligned}$
\item $X$ is $\mc K$-analytic $\iff$ $X$ has a complete sequence of countable covers.
\end{enumerate}
\end{proposition}
The equivalence between first and second part of $2.$ is easy, and follows from Lemma \ref{lemma: refinement}. The remaining assertion are due to Frolík (\cite{frolik1960generalizations}, \cite[Theorem 7]{frolik1963descriptive} and \cite[Theorem 9.3]{frolik1970survey}).  To get our main result, we will prove a statement which has a similar flavor:
\begin{proposition}\label{proposition: fs disjoint covers and abs fsd}
Any topological space with a complete sequence of countable disjoint $F_\sigma$ covers is absolutely $\fsd$.
\end{proposition}

We will need the following observation:
\begin{lemma}\label{lemma: refinement} Let $X$ be a topological space.
\begin{enumerate}
\item If $\left( \mc F_n \right)_{n\in\mathbb N}$ is a complete sequence of covers on $X$ and for each $n\in\mathbb N$, the cover $\mc F_n'$ is a refinement of $\mc F_n$, then the sequence $\left( \mc F_n' \right)_{n\in\mathbb N}$ is complete.
\item If $X$ has a complete sequence of countable closed (open, $F_\sigma$) covers, then it also has a complete sequence of countable closed (open, $F_\sigma$) covers $\left( \mc F_n \right)_{n\in\mathbb N}$, in which each $\mc F_{n+1}$ refines $\mc F_n$.
\end{enumerate}
\end{lemma}
\begin{proof}
The first part follows from the definition of complete sequence of covers of $X$. For the second part, let $\left( \mc F_n' \right)_{n\in\mathbb N}$ be a complete sequence of covers of $X$. We define the new sequence of covers as the refinement of $\left( \mc F_n' \right)_{n\in\mathbb N}$, setting $\mc F_1:=\mc F_1'$ and
\[ \mc F_{n+1}:=\mc F_{n+1}'\wedge\mc F_n := \left\{ F'\cap F|\ F'\in\mc F_n',\ F\in\mc F_{n+1} \right\}. \]
Clearly, the properties of being countable and closed (or $F_\sigma$) are preserved by this operation.
\end{proof}

The main reason for the use of complete sequences of covers is the following lemma:
\begin{lemma}\label{lemma: x is in X and c.s.of c.}
Let $\left( \mc F_n \right)_{n\in\mathbb N}$ be a complete sequence of covers of $X$ and $cX$ a compactification of $X$. If a sequence of sets $F_n\in\mc F_n$ satisfies $F_1\supset F_2\supset \dots$, then $\bigcap_{n \in\mathbb N}\overline{F_n}^{cX} \subset X$.
\end{lemma}
\begin{proof}
Fix $x\in cX$. We observe that the family
\[ \mc B:=\left\{ U \cap F_n|\ U\in\mc U\left(x\right),\ n\in\mathbb N \right\} \]
is, by hypothesis, formed by nonempty sets and closed under taking finite intersections, therefore it is a basis of some filter $\mc F$ (note that this is the only step where we use the monotonicity of $(F_n)_n$). Since every $F_n$ belongs to both $\mc F$ and $\mc F_n$, $\mc F$ must have some accumulation point $y$ in $X$. If $x$ and $y$ were distinct, they would have some neighborhoods $U$ and $V$ with disjoint closures. This would imply that $V\in\mc U\left (y\right )$, $U\supset U\cap F_1\in \mc F$ and $V\cap\overline{U}$, which contradicts the definition of $y$ being an accumulation point of $\mc F$. This means that $x$ is equal to $y$ and, in particular, $x$ belongs to $X$.
\end{proof}

The property of being hereditarily Lindelöf will be used in the following way:
\begin{lemma}\label{lemma: countable network and disjoint cover}
Every hereditarily Lindelöf $\fsd$ space $X$ has a complete sequence of countable disjoint $F_\sigma$ covers.
\end{lemma}
\begin{proof}
Recall that in a hereditarily Lindelöf space, every open set can be written as a countable union of closed sets. Consequently, the difference of two closed sets is always an $F_\sigma$ set.

Let $(\mc F_n )_n$ be a complete sequence of countable closed covers of $X$ (which exists by Proposition \ref{proposition: fsd iff complete sequence of covers}). We enumerate each of the covers $\mc F_n$ as $\mc F_n=\left\{ F^n_k| \ k\in\N \right\}$. Modifying each $\mc F_n$ in the standard way, we obtain disjoint covers of $X$:
\[ \widetilde {\mc F_n} := \left\{ F^n_k \setminus \left(F^n_1\cup\dots\cup F^n_{k-1} \right)|\ k\in\N \right\}.\]
By previous paragraph, each of these new covers consists of $F_\sigma$ sets. By $1.$ in Lemma \ref{lemma: refinement}, the sequence $\left( \widetilde {\mc F_n} \right)_n$ is complete.
\end{proof}

In order to get Theorem \ref{theorem: countable network}, it remains to prove Proposition \ref{proposition: fs disjoint covers and abs fsd}:
\begin{proof}[Proof of Proposition \ref{proposition: fs disjoint covers and abs fsd}]
Let $(\mc D_n)_n$ be the complete sequence of countable disjoint $F_\sigma$ covers of $X$. Without loss of generality, we can assume (by Lemma \ref{lemma: refinement}) that each $\mc D_{n+1}$ refines $\mc D_n$. Also, let $cX$ be a compactification of $X$. Since $cX$ is fixed, all closures will automatically be taken in this compactification.

We enumerate each cover as $\mc D_n=\left\{D^n_m|\ m\in\N \right\}$ and write each of its elements as countable union of closed sets: $D^n_m=\bigcup_i D^n_{m,i}$. We set $\widetilde { \mc D_n } := \left\{ D^n_{m,i} |\ m,i\in\N \right\}$ and $\widetilde { \mc D } :=\bigcup_n \widetilde { \mc D_n }$.

It is clear that $X\subset \bigcap_n \bigcup \left\{ \overline{D}|\ D\in \widetilde {\mc D_n} \right\}$. Note that the set on the right hand side is $\fsd$. The equality does not, in general, hold, but we can modify the right hand side using Lemma \ref{lemma: fsd spaces are Fs separated}.

Indeed, suppose that $x$ belongs to $\bigcap_n \bigcup \left\{ \overline{D}|\ D\in \widetilde {\mc D_n} \right\}$, but not to $X$. By Lemma \ref{lemma: x is in X and c.s.of c.}, this means that there are sequences $(m_n)_n$ and $(i_n)_n$, such that $x\in\bigcap_n \overline{D^n_{m_n,i_n}}$, but $D^1_{m_1,i_1}\supset D^2_{m_2,i_2}\supset \dots$ does \emph{not} hold. In particular, $x\in\overline D\cap \overline E$ holds for some $D,E\in\widetilde {\mc D}$ disjoint.

Since both $D$ and $E$ are closed in $X$, we have $\overline D \cap \overline E \subset cX \setminus X$. This means we can use Lemma \ref{lemma: fsd spaces are Fs separated} to obtain an $F_\sigma$ subset $H_{D,E}$ of $cX$ satisfying $X\subset H_{D,E}\subset cX\setminus \left( \overline D \cap \overline E \right)$. We claim that
\[ X = \bigcap_n \bigcup \left\{ \overline{D}|\ D\in \widetilde {\mc D_n} \right\} \cap \bigcap \left\{H_{D,E}|\  D,E\in\widetilde {\mc D} \textrm{ disjoint} \right\}. \]
By definition of $H_{D,E}$, the set on the right side contains $X$, and the  opposite inclusion follows from the observation above. Since $\widetilde {\mc D}$ is countable, the right hand side is $\fsd$. This proves that that $X\in\fsd \left(cX\right)$, which completes the proof (and also the whole section).
\end{proof}

\section{Conclusion}
We have shown that being hereditarily Lindelöf is a sufficient condition for an $\fsd$ space to be absolutely $\fsd$ - this is a fairly useful condition for applications.
The problem of finding the description of absolute $\fsd$ spaces remains yet unsolved, but we have gotten one step closer to the characterization: By Frolík's result, absolutely $\fsd$ space must have a complete sequence of countable $F_\sigma$ covers. If a space has such a sequence of covers which are also disjoint, then it must be absolutely $\fsd$. Therefore, if the desired characterization can be formulated in terms of complete sequences of covers, it must be something between these two conditions.

\bibliography{fsdrefs}
\end{document}